\documentclass[11pt,reqno]{amsart}

\usepackage[foot]{amsaddr}
\usepackage{adjustbox}
\usepackage{amssymb}
\usepackage{amsmath}
\usepackage{fancyhdr}
\usepackage[british]{babel}
\usepackage[a4paper,margin=1in,headheight=1cm,headsep=1cm,footskip=1cm]{geometry}
\usepackage{mathtools}
\usepackage{enumitem}
\usepackage{dsfont}
\usepackage{centernot}
\usepackage{xstring}
\usepackage{colortbl}
\usepackage[usenames,dvipsnames,table]{xcolor}
\usepackage{tikz}
\usepackage[
		bookmarksopen=true,
		bookmarksopenlevel=1,
		colorlinks=true,
		linkcolor=darkblue,
        linktoc=page,
		citecolor=darkblue,
]{hyperref}
\usepackage{bbm}
\usepackage{mathrsfs}
\usepackage{hyphenat}
\usepackage{soul}
\usepackage[capitalise]{cleveref}
\usepackage[noadjust]{cite}
\newcommand{\COMMENT}[1]{}

\newcommand{\APPENDIX}[1]{}
\newcommand{\NOAPPENDIX}[1]{#1}
\renewcommand{\APPENDIX}[1]{#1}\renewcommand{\NOAPPENDIX}[1]{}

\newcommand{\TASK}[1]{}
\renewcommand{\TASK}[1]{\footnote{\textcolor{red}{#1}}} 

\title{Hitting times in the binomial random graph}

\author[B.~Granet]{Bertille Granet}
\address{Institut f\"ur Informatik, INF 205, Universit\"at Heidelberg, 69120 Heidelberg, Germany}
\email{granet@informatik.uni-heidelberg.de}

\author[F.~Joos]{Felix Joos}
\email{joos@informatik.uni-heidelberg.de}

\author[J.~Schrodt]{Jonathan Schrodt}
\email{schrodt@informatik.uni-heidelberg.de}

\date{}

\thanks{Bertille Granet was supported by the Alexander von Humboldt Foundation.
Felix Joos and Jonathan Schrodt were partially supported by the Deutsche Forschungsgemeinschaft (DFG, German Research Foundation) -- 428212407.}

\newtheorem{theorem}{Theorem}[section]

\newtheorem{prop}[theorem]{Proposition}
\newtheorem{lemma}[theorem]{Lemma}
\newtheorem{lm}[theorem]{Lemma}
\newtheorem{cor}[theorem]{Corollary}

\theoremstyle{definition}

\newtheorem{claim}{Claim}
\newenvironment{proofclaim}[1][Proof of Claim]{\begin{proof}[#1]}{\end{proof}}

\newcounter{claimnumber}
\AtBeginEnvironment{proof}{\setcounter{claim}{0}}
\AtBeginEnvironment{proofclaim}{\setcounter{claimnumber}{\value{claim}}}
\AtEndEnvironment{proofclaim}{\setcounter{claim}{\value{claimnumber}}}

\newtheoremstyle{stepstyle}{10pt}{5pt}{\em}{0pt}{\em}{:}{5pt}{}
\theoremstyle{stepstyle}

\numberwithin{equation}{section}

\creflabelformat{equation}{#2(#1)#3}
\crefdefaultlabelformat{#2\textup{#1}#3}

\Crefname{enumi}{}{}
\Crefname{thm}{Theorem}{Theorems}
\Crefname{lm}{Lemma}{Lemmas}
\Crefname{cor}{Corollary}{Corollaries}
\Crefname{prop}{Proposition}{Propositions}
\Crefname{claim}{Claim}{Claims}
\Crefname{equation}{}{}
\Crefname{conjecture}{Conjecture}{Conjectures}
\Crefname{figure}{Figure}{Figures}
\Crefname{fact}{Fact}{Facts}

\definecolor{darkblue}{rgb}{0,0,0.5}

\allowdisplaybreaks

\def\op{\operatorname}

\newcommand{\pr}{\mathbb{P}}
\newcommand{\ex}{\mathbb{E}}
\newcommand{\bE}{\mathbb{E}}
\newcommand{\bP}{\mathbb{P}}

\newcommand{\bN}{\mathbb{N}}

\newcommand{\bR}{\mathbb{R}}

\newcommand{\cG}{\mathcal{G}}

\newcommand{\cO}{\mathcal{O}}
\newcommand{\cW}{\mathcal{W}}

\DeclareMathOperator{\Var}{Var}
\newcommand{\tr}{tr}
\newcommand{\1}{\mathbbm 1}

\usepackage{subfiles}

\begin{document}

\begin{abstract}
Fix $k\geq 2$, choose $\frac{\log n}{n^{(k-1)/k}}\leq p\leq 1-\Omega(\frac{\log^4 n}{n})$, and consider $G\sim G(n,p)$. For any pair of vertices $v,w\in V(G)$, we give a simple and precise formula for the expected number of steps that a random walk on $G$ starting at~$w$ needs to first arrive at~$v$. The formula only depends on basic structural properties of $G$.
This improves and extends recent results of Ottolini and Steinerberger, as well as Ottolini, who considered this problem for constant as well as for mildly vanishing $p$.
\end{abstract}

\maketitle
\section{Introduction}

Random walks on graphs have many applications in mathematics and physics, and are also in themselves an important subject of study in graph theory. See for example \cite{lovasz1993} for a survey covering various areas of research on random walks. Of particular interest is the following question: how long does it take for a random walk to hit a certain vertex? Hitting times were calculated for specific classes of (deterministic) graphs (see e.g.\ \cite{rao2013finding,palacios1994expected,palacios1990bounds,palacios2009hitting}), while more general results were obtained by Lov\'{a}sz~\cite{lovasz1993}, who gave an explicit formula for the hitting time in terms of the spectral properties of the host graph.

Another prolific avenue of research in graph theory is to solve general problems for the binomial random graph $G(n,p)$, that is, the graph on $n$ vertices where each edge is present independently with probability $p$. Löwe and Torres \cite{lowe2014hitting} used the above mentioned result of Lov\'asz to compute the expected time a random walk starting at a random vertex needs to hit a fixed vertex $v$.

More precisely, for a simple random walk~$X=X_0, X_1, \dots $ on a graph~$G$ starting at~$w\in V(G)$, 
denote by~$T_{wv}$ the \emph{hitting time} of $v$, that is, the number of steps the random walk needs to hit~$v$ for the first time. 
Denote by~$H_{wv}$ the expected hitting time of $v$, that is, $H_{wv}=\ex T_{wv}$.
Löwe and Torres \cite{lowe2014hitting} proved that there exists $\cO(1)\leq\xi\coloneqq\xi(n)\leq\cO(\log\log n)$ such that if $p=\omega(\frac{\log^\xi n}{n})$, then with high probability the random graph $G\sim G(n,p)$ satisfies $\sum_{w\in V(G)} \frac{2|E(G)|}{d(v)}H_{wv} = (1+o(1))n$ for any $v\in V(G)$.

However, this only gives a (weighted) average hitting time for a fixed $v$ across all possible starting points $w$, while we are interested in an explicit formula for $H_{wv}$ for any given $w,v$.
For $p\geq\frac{\log n}{\sqrt{n}}$ and $\ell\geq 3$, 
it is straightforward to compute that 
\[\pr[X_\ell = v \mid X_0 = w] = \frac{1}{n} \pm \cO\left(\frac{\sqrt{\log n}}{pn^{3/2}}\right)\]
(see \cref{lemma: random walk mixing time}). This implies that $H_{wv}\approx n$ for any $w\neq v\in V(G)$.%
\COMMENT{Let~$C$ such that $\pr[X_3=v\mid X_0=w] = \frac{1}{n} \pm \frac{C\sqrt{\log n}}{pn^{3/2}}$ for any $w\in V(G)$. Let $q\coloneqq \frac{1}{n}-\frac{C\sqrt{\log n}}{pn^{3/2}}$. For each $j\in\bN$, let $Y_j=1$ if $\pr[X_{3j}=v\mid X_0=w]$ and $Y_j=0$ otherwise. Then $Y_1,Y_2,\ldots$ stochastically dominates a sequence $Z_1,Z_2,\ldots$ of $\op{Ber}(q)$-distributed $\{0,1\}$-random variables. As the expected time~$i$ where $Z_i=1$ for the first time is $\frac{1}{q}$, the statement follows.}
For constant~$p$ this has recently been improved by Ottolini and Steinerberger~\cite{ottolini2023concentration}, who gave an explicit formula for the hitting time up to a small error term which tends to $0$ as $n$ tends to infinity.

More precisely, fix $p\in(0,1)$ and let $G\sim G(n,p)$. 
It is well known that, with high probability, $G$~has diameter~2.
Let $v\in V(G)$ and let $H_{N(v)}\coloneqq\frac{1}{|N(v)|}\sum_{u\in N(v)}H_{uv}$ be the average expected hitting time in~$N(v)$. 
As observed in \cite{ottolini2023concentration}, it is not difficult to show that $H_{N(v)}=\frac{2|E(G)|}{d(v)}-1$ (see Lemma~\ref{lemma: H_nuv}).
Ottolini and Steinerberger \cite{ottolini2023concentration} then proved that, for any $w\in N(v)$, the value $H_{wv}$ is essentially equal to the average hitting time in~$N(v)$.
That is, it is irrelevant at which vertex $w\in N(v)$ the random walk starts, the expected time to hit $v$ is every time essentially the same.
If $\op{dist}(v,w)=2$, one needs to account for the average time a random walk needs to pass to~$N(v)$. As the probability of moving from a vertex of distance~2 to a vertex of distance~1 is about~$p$, the expected time to pass to~$N(v)$ is roughly~$\frac{1}{p}$.

\begin{theorem}[Ottolini and Steinerberger~\cite{ottolini2023concentration}]\label{thm: ottolini steinerberger}
    Let $p\in(0,1)$ and $G\sim G(n,p)$.
    Let $v\neq w\in V(G)$. Then, with high probability,
    \begin{equation*}
        H_{wv} = \frac{2|E(G)|}{d(v)} - 1 + \frac{1}{p}\1_{\op{dist}(v,w)=2} \pm \cO\left(\frac{\log^{3/2} n}{\sqrt{n}}\right).
    \end{equation*}
\end{theorem}

In a follow-up article,  Ottolini \cite{ottolini2024blockmodel} derived a similar result for any~$p$ with $\frac{\log^5n}{p^8n}\to0$, extending~\cref{thm: ottolini steinerberger} to the case where~$p$ vanishes mildly.
Observe that in this case, $G$ has diameter~$2$ as in Theorem~\ref{thm: ottolini steinerberger}.

To prove Theorem~\ref{thm: ottolini steinerberger}, Ottolini and Steinerberger~\cite{ottolini2023concentration} rely on the fact that $|N(v)\cap N(w)|\approx p^2n$ for any $v,w\in V(G)$. 
For small enough~$p$, this is not necessarily true. However, as long as~$p\geq n^{-1+o(1)}$, it is still possible to find an easy and precise formula for~$H_{wv}$. This is our main result.

For any graph $G$, vertices $v,w\in V(G)$, and $i\in \bN$, let $\cW_i(w)$ be the number of walks of length~$i$ starting at~$w$ and let~$\cW_i(w,v)$ be the number of $w$-$v$ walks of length~$i$.

\begin{theorem}\label{thm: main}
    Let $k\geq 2$, let $\frac{\log n}{n^{(k-1)/k}}\leq p\leq 1-\Omega(\frac{\log^4n}{n})$,
    and $G\sim G(n,p)$. Let $v\neq w\in V(G)$. Then, with high probability,%
    \begin{align*}
        H_{wv} &= \frac{2|E(G)|}{d(v)} - 1 + \frac{2|E(G)|}{d(v)} \frac{1}{|N(v)|} \sum_{u\in N(v)}\sum_{i=1}^{3k+1}\left(\frac{\cW_i(u,v)}{\cW_i(u)}-\frac{\cW_i(w,v)}{\cW_i(w)}\right)
        \pm\cO\left(\frac{\sqrt{\log n}}{p^{3/2}\sqrt{n}}\right).
    \end{align*}
\end{theorem}

The missing range of $1-\Omega(\frac{\log^4n}{n})< p\leq 1$ is an artefact of our proof which uses, as a black box, a spectral result which does not cover the cases when $p$ is close to $1$ (see \cref{app:lambda2} for details). 

The error bound can be improved to $\cO\left(\frac{\sqrt{\log n}}{\sqrt{pn}}\right)$ when one considers the probability that the random walk reaches~$v$ after~$i$ steps rather than the proportion of walks of length~$i$ which end at $v$ (see \cref{thm: main version with probabilities}). When $k=2$, the diameter of~$G$ is~2 (with high probability) and these probabilities can be explicitly calculated (up to some small error term). This implies the following result, which in turn implies~\cref{thm: ottolini steinerberger} as well as its corresponding result in~\cite{ottolini2024blockmodel}, and improves on the error term.

\begin{cor}\label{cor: diameter 2}
    Let $\frac{\log n}{\sqrt{n}}\leq p\leq 1-\Omega(\frac{\log^4n}{n})$, and $G\sim G(n,p)$. Let $v\neq w\in V(G)$. Then, with high probability,
    \begin{equation*}
        H_{wv} = \frac{2|E(G)|}{d(v)} - 1 + \frac{1}{p}\1_{\op{dist}(v,w)=2} \pm\cO\left(\frac{\sqrt{\log n}}{p^{3/2}\sqrt{n}}\right).
    \end{equation*}
\end{cor}

\section{Proof Outline}

Let $k\geq 2$ be an integer and let $\frac{\log n}{n^{(k-1)/k}}\leq p\leq 1-\Omega(\frac{\log^4n}{n})$. Let $G\sim G(n,p)$ and note that, throughout our proof, we will only make use of standard properties of the random graphs, as well as well-known spectral properties.
Most of the ideas presented below were already used by Ottolini and Steinerberger \cite{ottolini2023concentration} to prove \cref{thm: ottolini steinerberger}. We discuss the differences and our contributions more precisely at the end of this section.

The basic idea to find a formula for~$H_{wv}$ is to compare~$H_{wv}$ to $H_{N(v)}\coloneqq \frac{1}{|N(v)|}\sum_{u\in N(v)}H_{uv}$, the average hitting time in~$N(v)$. We have
\begin{equation*}
    H_{wv} = H_{N(v)} + \frac{1}{|N(v)|} \sum_{u\in N(v)} (H_{wv}-H_{uv}).
\end{equation*}
Explicitly calculating~$H_{N(v)}$ and the difference of expected hitting times depending on the starting vertex of the random walk yields the desired result.

We give a short outline of how these calculations work.
The distribution of a random walk on~$G$ after~$\ell$ steps converges rapidly to a stationary distribution~$\pi$, where $\pi(v) = \frac{d(v)}{2|E(G)|}$ for each $v\in V(G)$. It is a well-known fact that the mean return time of a random walk to a vertex~$v\in V(G)$ is given by $\frac{1}{\pi(v)} = \frac{2|E(G)|}{d(v)}$.
As the mean return time is $1 + \sum_{u\in N(v)}\pr[X_1=u\mid X_0=v]H_{uv} = 1+\frac{1}{|N(v)|}\sum_{u\in N(v)}H_{uv} = 1+ H_{N(v)}$, this implies that $H_{N(v)}=\frac{2|E(G)|}{d(v)}-1$.

To compare~$H_{wv}$ and~$H_{w'v}$ for $w,w'\in V(G)$, we look at the expected hitting times after a small number of steps.
If the random walk has not yet hit~$v$, it essentially makes no difference where the random walk started. The only case to consider is that the random walk may actually hit~$v$ in the first few steps. By explicitly calculating these differences, we obtain our main result.

As mentioned, although these three core ideas are essentially borrowed from Ottolini and Steinerberger \cite{ottolini2023concentration}, we make non-trivial adjustments to extend the results from constant $p$ to the much wider range corresponding to constant diameter.
Moreover, while the core idea of comparing $H_{wv}$ to average hitting times was used in \cite{ottolini2023concentration}, we implement it differently. In \cite{ottolini2023concentration}, vertices at distance $1$ and $2$ from $v$ were in fact treated separately: $H_{wv}$ was compared to $H_{N(v)}$ when $w\in N(v)$ and to $\frac{1}{n-|N(v)\cup \{v\}|}\sum_{u\notin N(v)\cup \{v\}}H_{uv}$ otherwise. Here, we always compare $H_{wv}$ to $H_{N(v)}$, which not only requires non-trivial adjustments, but in fact leads to a simpler proof and better error terms.

\section{Preliminaries}

\subsection{Notation}

For $\alpha,\beta,\delta\in \mathbb{R}$, we write $\alpha=(1\pm\delta)\beta$ if $(1-\delta)\beta\leq\alpha\leq(1+\delta)\beta$.
All asymptotic notation is understood with respect to~$n$.
For functions $f$ and $g$ we sometimes write $f\lesssim g$ instead of $f=\cO(g)$ and $f\gtrsim g$ instead of $g=\cO(f)$.
We set $\bN\coloneqq\{1,2,\ldots\}$ and $\bN_0\coloneqq\bN\cup\{0\}$, and for any $k\in \mathbb{N}$ we write $[k]\coloneqq \{1,\dots, k\}$.

Let $G$ be a graph. For $v,w\in V(G)$, we denote by $\op{dist}(v,w)$ the length of the shortest path between~$v$ and~$w$ and by $\op{diam}(G)$ the diameter of $G$. If we take a sum over all vertices of a graph~$G$ we often write~$\sum_{v}$ instead of~$\sum_{v\in V(G)}$ and, given $w\in V(G)$, $\sum_{v\neq w}$ instead of $\sum_{v\in V(G)\setminus\{w\}}$.

All random walks on a graph are to be understood as simple random walks. Let~$G$ be a graph and $v,w\in V(G)$. Let $X=X_0,X_1,\ldots$ be a random walk on~$G$ starting at $w$. We denote $\pr_w[\cdot]\coloneqq\pr[\cdot\mid X_0=w]$ and $\ex_w[\cdot]\coloneqq\ex[\cdot\mid X_0=w]$. When $X$ has a unique stationary distribution, we denote it by $\pi$. For $\ell\in \bN_0$, we sometimes write $\mu_{\ell,w}(\cdot)\coloneqq \pr_w[X_\ell=\cdot]$ for the probability distribution of the random walk after $\ell$ steps. If~$w$ is clear from the context, we may omit~$w$ and write simply~$\mu_\ell$. Let~$T_{wv}$ be the time the random walk takes to hit~$v$ for the first time and $H_{wv}\coloneqq\ex T_{wv}$. For a probability distribution~$\delta$ on $V(G)$, we define $H_{\delta v}\coloneqq\sum_{u}\delta(u)H_{uv}$.

All vectors in this paper are row vectors. Given a vector $x=(x_1,\dots, x_k)$, we set $\|x\|\coloneqq \sum_{i\in [k]}|x_i|$. Probability distributions are sometimes viewed as vectors.

Let $G$ be a graph on $n$ vertices. We denote by $A$ its adjacency matrix and by $\lambda_1\geq \dots \geq \lambda_n$ the eigenvalues of $A$. We assume without loss of generality that the eigenvectors associated to $\lambda_1, \dots, \lambda_n$ form an orthonormal basis and denote by $\phi=(\phi_1, \dots, \phi_n)$ the eigenvector corresponding to $\lambda_1$.

The following statement is useful for our calculations. We omit its straightforward proof.

\begin{lemma}\label{lemma: landau}
    Let $f,f',g$, and $g'$ be functions such that $g'= \cO(g)$. Suppose that $f'g=\cO(fg')$ or $fg'=\cO(f'g)$. Then,%
    \COMMENT{We calculate
        \begin{align*}
            \frac{f\pm\cO(f')}{g\pm\cO(g')} - \frac{f}{g}
            &= \frac{fg\pm\cO(f'g) - (fg\pm\cO(fg'))}{g(g\pm\cO(g'))}
            = \frac{\pm\cO(f'g)\pm\cO(fg')}{g^2\pm\cO(gg')}
        \end{align*}
        and the statement follows from $gg'= \cO(g^2)$.
    }
    \begin{equation*}
        \frac{f\pm\cO(f')}{g\pm\cO(g')} =
        \frac{f}{g} \pm \left\{
        \begin{matrix*}[l]
            \cO\left(\frac{fg'}{g^2}\right) & \text{if $f'g= \cO(fg')$}, \\
            \cO\left(\frac{f'}{g}\right) & \text{if $fg'= \cO(f'g)$}.
        \end{matrix*}\right.
    \end{equation*}
\end{lemma}

\subsection{Properties of $G(n,p)$}

We make use of the following standard concentration inequality.%
\COMMENT{For $0<\delta\leq 1$, \cite[Theorem 2.1]{janson2011random} gives
\[\pr[Y \geq \ex Y+\delta\ex Y]\leq \exp\left(-\frac{\delta^2 (\ex Y)^2}{2(1+\delta/3)\ex Y}\right)\leq \exp\left(-\frac{\delta^2\ex Y}{3}\right)\]
and
\[\pr[Y \leq \ex Y- \delta\ex Y]\leq \exp\left(-\frac{\delta^2 (\ex Y)^2}{2\ex Y}\right)\leq \exp\left(-\frac{\delta^2\ex Y}{3}\right),\]
as desired.
For $\delta>1$, we clearly have $\pr[Y \leq \ex Y- \delta\ex Y]=0$ and, by \cite[Theorem 2.1]{janson2011random}, we have
\[\pr[Y \geq \ex Y+\delta\ex Y]\leq \exp\left(-\ex Y(1+\delta)\log(1+\delta)\right)\leq \exp\left(-\frac{\delta\ex Y}{3}\right).\]}

\begin{lemma}[Chernoff bound (see e.g.\ {\cite[Theorem 2.1]{janson2011random}}]\label{lemma: Chernoff}
    Let $Y$ be a binomial random variable. Then, for any $\delta>0$,
    \begin{equation*}
        \pr[|Y - \ex Y| \geq \delta\ex Y] \leq 2\exp\left(-\frac{\min\{\delta,\delta^2\}}{3}\ex Y\right).
    \end{equation*}
\end{lemma}

We now list a number of standard properties for a typical $G\sim G(n,p)$.
\begin{lemma}\label{lemma: properties of G(np)}
    Let $k\geq 2$, let $\frac{\log n}{n^{(k-1)/k}}\leq p\leq 1-\Omega(\frac{\log^4n}{n})$, and $G\sim G(n,p)$. Then, with high probability, $G$ satisfies all of the following properties.
    \begin{enumerate}[label=\rm(\roman*)]
        \item $G$ is not bipartite. \label{item: G not bipartite}
        \item $\op{diam}(G) \leq k$. \label{item: diam G}
        \item For all $v\in V(G)$, $d(v) = pn \pm \cO(\sqrt{pn\log n})$. \label{item: N(v)}
        \item $2|E(G)| = pn^2 \pm \cO(\sqrt{pn^2\log n})$. \label{item: E(G)}
        \item For all $v\neq w\in V(G)$, $|N(v)\cap N(w)| = p^2n \pm \cO(\max\{\sqrt{p^2n\log n},\log n\})$. \label{item: N(v) cap N(w)}
        \item For all $i\in [n]$, $\phi_i=\frac{1}{\sqrt{n}}\pm \cO\left(\frac{\log^{3/2} n}{\sqrt{p}n\log (pn)}\right)$.\label{item:phi}
        \item $\lambda_1=(1+o(1))pn$.\label{item:lambda1}
        \item $\max\{|\lambda_2|,|\lambda_n|\}=\cO(\sqrt{pn})$.\label{item:lambda2}
    \end{enumerate}
\end{lemma}
\begin{proof}
    First, note that \cite[Theorem 1]{mitra2009entrywise} states that \cref{item:phi} holds with high probability, while \cite[Theorem 1.1]{krivelevich2003largest} states that \cref{item:lambda1} holds with high probability when \cref{item: N(v)} is satisfied.%
        \COMMENT{\cite[Theorem 1.1]{krivelevich2003largest} gives $\lambda_1=(1+o(1))\max\{\sqrt{\Delta(G)},pn\}$. When \cref{item: N(v)} holds, $\sqrt{\Delta(G)}\lesssim \sqrt{pn}$ so $\sqrt{\Delta(G)}\leq pn$.}
    Moreover, \cite[Theorem 1]{furedi1981eigenvalues} states that \cref{item:lambda2} holds with high probability when $p$ is constant. The key part of their argument was shown to work in our range of $p$ by Vu \cite{vu2007spectral}, see \APPENDIX{\cref{app:lambda2}}\NOAPPENDIX{the appendix in the arXiv version of this paper} for more details.
    Therefore, it suffices to show that each of~\cref{item: G not bipartite,item: N(v) cap N(w),item: diam G,item: N(v),item: E(G)} holds with high probability. Applying a union bound yields the desired result.

    Note that~\ref{item: G not bipartite} follows from the well-known fact that~$G$ contains a triangle, while~\ref{item: diam G} follows from~\cite[Corollaries 7 and 8]{bollobas1981diameter}.%
        \COMMENT{For any $G'\subseteq G$, we have $\op{diam}(G')\geq \op{diam}(G)$. So for any $p'\leq p$, we have $\op{diam}(G(n,p'))\geq \op{diam}(G(n,p))$. 
        Let $p'\coloneqq \frac{\log n}{n^{(k-1)/k}}$. By \cite[Corollaries 7 and 8]{bollobas1981diameter}, we have $\op{diam}(G(n,p'))=k$ whp.}

    For~\ref{item: N(v)}, let $v\in V(G)$. Then $\ex d(v) = p(n-1)$
    and so \cref{lemma: Chernoff} gives\COMMENT{Set $\delta\coloneqq\frac{3\sqrt{\log n}}{\sqrt{\ex d(v)}}<1$. Then 
    \begin{align*}
        \pr[|d(v)-pn|\geq 4\sqrt{pn\log n}]
        &\leq \pr[|d(v)-p(n-1)|\geq 3\sqrt{p(n-1)\log n}] \\
        &= \pr[|d(v)-\ex d(v)| \geq \delta \ex d(v)] \leq 2\exp\left(-\frac{\delta^2}{3}\ex d(v)\right) =2n^{-3}.
    \end{align*}}
    \[\pr[|d(v)-pn|\geq 4\sqrt{pn\log n}] \leq 2n^{-3}.\]
    We apply a union bound to see that $d(v)=pn\pm \cO(\sqrt{pn\log n})$ holds for all $v\in V(G)$ with probability at least $1-2n^{-2}$.

    Analogously we see that~\ref{item: E(G)} holds with probability $1-o(1)$.%
        \COMMENT{$2\ex E(G)=p(n-1)n$ so \cref{lemma: Chernoff} with $\delta\coloneqq \frac{2\sqrt{\log n}}{\sqrt{2\ex E(G)}}<1$ implies
        \begin{align*}
            \pr[|2E(G)- pn^2|\geq  3 \sqrt{pn^2\log n}]&\leq \pr[|2E(G)-p(n-1)n|\geq 2\sqrt{p(n-1)n\log n}]\\
            &=\pr[|2E(G)- 2\ex E(G)|\geq \delta 2\ex E(G)]
            \leq 2\exp\left(-\frac{2\delta^2 \ex E(G)}{3}\right)=2n^{-4/3}.
        \end{align*}}

    For~\ref{item: N(v) cap N(w)}, let $v\neq w\in V(G)$. Then $\ex |N(v)\cap N(w)| = p^2(n-2)$.
    If $p^2n\leq 10\log n$, then
    \cref{lemma: Chernoff} yields%
        \COMMENT{Note $\delta\coloneqq \frac{9\log n}{\ex |N(v)\cap N(w)|} \geq 1$, thus~\cref{lemma: Chernoff} yields
        \begin{align*}
            \pr[||N(v)\cap N(w)| - p^2n| \geq 10\log n]\leq \pr[||N(v)\cap N(w)| - p^2(n-2)| \geq 9\log n] \leq 2\exp(-3\log n).
        \end{align*}
        }
    \[\pr[||N(v)\cap N(w)| - p^2n| \geq 10\log n] \leq 2n^{-3}.\]
    If $p^2n > 10\log n$, then \cref{lemma: Chernoff} yields%
        \COMMENT{Note $\delta\coloneqq\frac{3\sqrt{\log n}}{\sqrt{\ex |N(v)\cap N(w)|}}<1$ so
        \begin{align*}
            \pr[||N(v)\cap N(w)| - p^2n| \geq 4\sqrt{p^2n\log n}]&\leq \pr[||N(v)\cap N(w)| - p^2(n-2)| \geq 3\sqrt{p^2(n-2)\log n}]\\
            &=\pr[||N(v)\cap N(w)| - \ex |N(v)\cap N(w)|| \geq \delta \ex |N(v)\cap N(w)|]\\
            &\leq 2\exp(-3 \log n)= 2n^{-3}.
    \end{align*}}
    \[\pr[|N(v)\cap N(w)| - p^2n| \geq 4\sqrt{p^2n\log n}] \leq 2n^{-3}.\]
    We apply a union bound to see that $|N(v)\cap N(w)|=p^2n\pm \cO(\max\{\sqrt{p^2n\log n},\log n\})$ holds for all $v\neq w\in V(G)$ with probability at least $1-2n^{-1}$.
\end{proof}

Given $k,n\in \bN$ with $k\geq 2$, and $\frac{\log n}{n^{(k-1)/k}}\leq p\leq 1-\Omega(\frac{\log^4n}{n})$, we denote by~$\cG_{k,n,p}$ the set of all graphs that satisfy the properties of \cref{lemma: properties of G(np)}. For simplicity, when we write $G\in \cG_{k,n,p}$, we always assume implicitly that $k,n$, and $p$ satisfy the required conditions. In particular, any parameter which is fixed before choosing $G\in \cG_{k,n,p}$ is understood as being independent of $n$. Note that we will fact prove \cref{thm: main,cor: diameter 2} for any (deterministic) graph in~$\cG_{k,n,p}$.

We state the following direct corollary describing key quantities for our proofs later. We omit its straightforward proof.

\begin{cor}\label{cor: properties of G(np)}
    Let $G\in\cG_{k,n,p}$ and $v\in V(G)$. Then, the following hold.%
        \COMMENT{For~\ref{item: 1/N(v)}, note that $|N(v)|=pn\pm\cO(\sqrt{pn\log n})$. By \cref{lemma: landau}, $\frac{1}{pn\pm\cO(\sqrt{pn\log n})} = \frac{1}{pn}\pm\cO(\frac{1\cdot\sqrt{pn\log n})}{p^2n^2})$. For~\ref{item: 2E/deg v}, note that also $2|E(G)|=pn^2\pm\cO(\sqrt{pn^2\log n})$. By \cref{lemma: landau}, $\frac{pn^2\pm\cO(\sqrt{pn^2\log n})}{pn\pm\cO(\sqrt{pn\log n})} = n \pm \cO(\frac{pn^2\sqrt{pn\log n}}{p^2n^2})$. For~\ref{item: deg v/2E}, we use~\ref{item: 2E/deg v}. \Cref{lemma: landau} yields $\frac{1}{n\pm\cO(\frac{\sqrt{n\log n}}{\sqrt{p}})} = \frac{1}{n} \pm \cO(\frac{\sqrt{n\log n}}{\sqrt{p}n^2})$.
        For~\ref{item: stationary distribtution}, note that~$X$ is irreducible, as~$G$ is connected and aperiodic, as~$G$ is not bipartite. It is a well-known fact that an irreducible aperiodic discrete Markov chain has a stationary distribution and the fact that it is given by the equation above can just be checked by doing a simple calculation.
        Finally, \cref{item: Twv} follows from \cref{item: 1/N(v)}. Indeed, this is clearly true for $i=1$. If $i>1$, we have $\pr_w[X_i=v] = \pr[X_i = v \mid X_{i-1}\in N(v)]\pr_w[X_{i-1}\in N(v)] \leq \frac{2}{pn}\cdot 1$.
        }
    \begin{enumerate}[label=\rm(\roman*)]
        \item $\frac{1}{d(v)} = \frac{1}{pn} \pm \cO(\frac{\sqrt{\log n}}{p^{3/2}n^{3/2}})$. \label{item: 1/N(v)}
        \item $\frac{2|E(G)|}{d(v)} = n \pm \cO(\frac{\sqrt{n\log n}}{\sqrt{p}})$. \label{item: 2E/deg v}
        \item $\frac{d(v)}{2|E(G)|} = \frac{1}{n} \pm \cO(\frac{\sqrt{\log n}}{\sqrt{p}n^{3/2}})$. \label{item: deg v/2E}
        \item Any random walk~$X$ on~$G$ is irreducible and aperiodic. In particular, there exists a unique stationary distribution~$\pi$ of~$X$ and $\pi(v)=\frac{d(v)}{2|E(G)|} = \frac{1}{n}\pm\cO(\frac{\sqrt{\log n}}{\sqrt{p}n^{3/2}})$. \label{item: stationary distribtution}
        \item For any $i\in \bN$ and any random random walk $X=X_0, X_1, \dots$ starting at $w\in V(G)$, we have $\bP[T_{wv}=i]\leq \bP[X_i=v]\leq \frac{2}{pn}$.\label{item: Twv}
    \end{enumerate}
\end{cor}

Let $G\in\cG_{k,n,p}$ and $v\in V(G)$. The next lemma explicitly states the average hitting time of all vertices in~$N(v)$. This will be useful to derive an explicit hitting time formula in \cref{lemma: formula for Hitting times}.

\begin{lemma}[{\cite[Lemma~1]{ottolini2023concentration}}]\label{lemma: H_nuv}
    Let $G\in\cG_{k,n,p}$ and $v\in V(G)$. Then,
    \begin{equation*}
        H_{{N(v)}} \coloneqq \frac{1}{|N(v)|}\sum_{u\in N(v)} H_{uv} = \frac{2|E(G)|}{d(v)} - 1.
    \end{equation*}
\end{lemma}

\subsection{Mixing times}

In \cite{ottolini2023concentration}, Ottolini and Steinerberger showed that, for constant $p$, a random walk on $G(n,p)$ converges rapidly to the stationary distribution. Their proof made in particular use of the spectral properties of a random graph. Using the corresponding properties for our range of $p$, that is, \cref{lemma: properties of G(np)}\cref{item:lambda1,item:lambda2,item:phi}, we can derive the following analogously. \APPENDIX{See \cref{app:mixing} for more details.}\NOAPPENDIX{See the appendix in the arXiv version of this paper for more details.}

Recall that a random walk $X$ on $G\in \cG_{k,n,p}$ starting at a fixed $w\in V(G)$ has a unique stationary distribution (\cref{cor: properties of G(np)}\cref{item: stationary distribtution}), which we denote by $\pi$, and that the probability distribution of $X$ at any stage $\ell\in \bN$ is denoted by $\mu_\ell$.

\begin{prop}[{\cite[Proposition]{ottolini2023concentration}}]\label{prop: mixing property}
    Let $\ell\in \bN$ and $G\in\cG_{k,n,p}$. Let $X=X_0,X_1,\ldots$ be a random walk on~$G$ starting at $w\in V(G)$. Then,
    \begin{equation*}
        \|\mu_\ell-\pi\| \lesssim \frac{(\log n)^{(\ell-1)/2}\sqrt{n}}{(pn)^{\ell/2}}.
    \end{equation*}
\end{prop}

\begin{cor}\label{lemma: random walk probabilities}
    Let $k\geq 2$ and $\ell\geq 3k+2$ be integers. Let $G\in \cG_{k,n,p}$ and $X= X_0,X_1,\ldots$ be a random walk on~$G$ starting at $w\in V(G)$. Then, for any $x\in V(G)$,
    \begin{equation*}
        \pr_w[X_\ell=x] = \frac{1}{n} \pm \cO\left(\frac{\sqrt{\log n}}{\sqrt{p}n^{3/2}}\right).
    \end{equation*}
\end{cor}
\begin{proof}
    Applying \cref{cor: properties of G(np)}\ref{item: stationary distribtution} and~\cref{prop: mixing property} yields%
    \COMMENT{
        \begin{align*}
            \sqrt{p}\|\mu_{\ell}-\pi\|
            &\lesssim \frac{(\log n)^{(\ell-1)/2}\sqrt{pn}}{(pn)^{\ell/2}}
            = \left(\frac{\log n}{pn}\right)^{(\ell-1)/2}
            \leq \left(\frac{1}{n^{1/k}}\right)^{(\ell-1)/2}
            \leq \frac{1}{n^{3/2+1/2k}}
            \leq \frac{\sqrt{\log n}}{n^{3/2}}.
        \end{align*}
    }
    \begin{align*}
        \left|\mu_\ell(x)-\frac{1}{n}\right|
        &\leq |\mu_\ell(x)-\pi(x)|+\left|\pi(x)-\frac{1}{n}\right|
        \lesssim \|\mu_\ell-\pi\| + \frac{\sqrt{\log n}}{\sqrt{p}n^{3/2}}
        \lesssim \frac{\sqrt{\log n}}{\sqrt{p}n^{3/2}},
    \end{align*}
    as desired.
\end{proof}

\section{Differences in hitting times}
The aim of this section is to show the following result, which shows that after walking a certain number of steps, the expected hitting time is only marginally influenced by the starting vertex.

\begin{lemma}\label{lemma: H_muv-Hmu'v}
    Let $k\geq 2$ and $\ell\geq 3k+2$ be integers. Let $G\in\cG_{k,n,p}$ and $v,w,w'\in V(G)$.
    Then,
    \begin{equation*}
        |H_{\mu_{\ell,w} v} - H_{\mu_{\ell,w'} v}| \lesssim \frac{\sqrt{\log n}}{\sqrt{pn}}.
    \end{equation*}
\end{lemma}
We prove this result by finding a simple upper bound for $|H_{wv}-n|$ and using the fact that $\|\mu_{\ell,w}-\pi\|$ tends to 0 quickly.
To find an upper bound for $|H_{wv}-n|$, we need the following lemma, which states that a random walk on $G\in\cG_{k,n,p}$ starting at~$w$ is expected to hit~$v$ in~$\cO((pn)^k)$ steps.

\begin{lemma}\label{lemma: H_wv lesssim n}
    Let $G\in\cG_{k,n,p}$ and $v,w\in V(G)$. Then, $H_{wv}\lesssim (pn)^k$.
\end{lemma}
\begin{proof}
    Let $X = X_0,X_1,\ldots$ denote a random walk on~$G$ starting at~$w$.
    Note that for all $u\in V(G)$, there is a $u$-$v$ path of length $\op{dist}(u,v)$ in~$G$.
    Note further that, for any $\ell\in\bN$ and any adjacent $u',u''\in V(G)$, we have $\pr[X_\ell = u'' \mid X_{\ell-1}=u']\gtrsim \frac{1}{pn}$ by~\cref{cor: properties of G(np)}\cref{item: 1/N(v)}. Thus, \cref{lemma: properties of G(np)}\cref{item: diam G} implies that $\pr_u[X_{\op{dist}(u,v)}=v] \gtrsim \frac{1}{(pn)^{\op{dist}(u,v)}} \geq \frac{1}{(pn)^k}$ for all $u\in V(G)$.
    Let~$c$ be a constant such that $\pr_u[X_{\op{dist}(u,v)}=v] \geq \frac{c}{(pn)^k}\eqqcolon q$ for any $u\in V(G)$.
    This implies that for any $j\in\bN_0$, we have $\pr_w[X_i =v \text{ for some } jk<i\leq (j+1)k]\geq q$.\COMMENT{Markov property: $\pr_w[X_{jk+\ell}=v] = \pr_u[X_\ell=v]$, where~$u$ is the vertex with $X_{jk}=u$.}
    For each $j\in \bN_0$, let $Y_j=1$ if $X_i=v$ for some $jk<i\leq (j+1)k$ and let $Y_j=0$ otherwise. Then $Y_0,Y_1,\ldots$ stochastically dominates a sequence $Z_0,Z_1,\ldots$ of independent $\op{Ber}(q)$-distributed $\{0,1\}$-random variables.
    As the expected time~$i$ where $Z_i=1$ for the first time is $\frac{1}{q}-1$, the statement follows.
\end{proof}

The following corollary quantifies how far $H_{wv}$ deviates from $n$. It follows directly from~\cref{lemma: H_wv lesssim n} and the fact that $(pn)^k\geq n$.

\begin{cor}\label{cor: H_wv-n}
   Let $G\in \cG_{k,n,p}$ and $v\neq w\in V(G)$. Then, $|H_{wv}-n| \lesssim (pn)^{k}$.
\end{cor}

We are now equipped to prove the main result of this section.
\begin{proof}[Proof of \cref{lemma: H_muv-Hmu'v}]
    Using \cref{cor: H_wv-n} we obtain
    \begin{align*}
        |H_{\mu_{\ell,w}v} - H_{\mu_{\ell,w'}v}|
        &= \left|\sum_{u\neq v} (\mu_{\ell,w}(u) - \mu_{\ell,w'}(u)) H_{uv}\right|
        = \left|\sum_{u\neq v}(\mu_{\ell,w}(u)-\mu_{\ell,w'}(u))(n+H_{uv}-n) \right| \\
        &\leq n\left|\sum_{u\neq v}(\mu_{\ell,w}(u)-\mu_{\ell,w'}(u))\right| + \sum_{u\neq v}|\mu_{\ell,w}(u)-\mu_{\ell,w'}(u)||H_{uv}-n| \\
        &\lesssim n|(1-\mu_{\ell,w}(v))-(1-\mu_{\ell,w'}(v))|+ (pn)^{k} \sum_{u\neq v}|\mu_{\ell,w}(u)-\mu_{\ell,w'}(u)| \\
        &\leq n|\mu_{\ell,w}(v)-\mu_{\ell,w'}(v)| + (pn)^{k} \|\mu_{\ell,w}-\mu_{\ell,w'}\|.
    \end{align*}
    It remains to show that both of these terms are asymptotically bounded by $\frac{\sqrt{\log n}}{\sqrt{pn}}$.
    First, \cref{lemma: random walk probabilities} yields $n|\mu_{\ell,w}(v)-\mu_{\ell,w'}(v)|\lesssim \frac{\sqrt{\log n}}{\sqrt{pn}}$, as desired.
    Second, \cref{prop: mixing property} yields
    \[
        \|\mu_{\ell,w}-\mu_{\ell,w'}\| \leq \|\mu_{\ell,w}-\pi\| + \|\pi-\mu_{\ell,w'}\| \lesssim \frac{(\log n)^{(\ell-1)/2}\sqrt{n}}{(pn)^{\ell/2}},
    \]
    and thus, as~$\ell\geq 3k+2$,%
        \COMMENT{We have
            \begin{align*}
                \sqrt{pn}(pn)^{k} \|\mu_\ell-\mu_\ell'\|
                &\lesssim \frac{(\log n)^{(\ell-1)/2}\sqrt{n}}{(pn)^{((\ell-1)/2)-k}}
                \leq \frac{(\log n)^{(\ell-1)/2}\sqrt{n}}{(\log n)^{((\ell-1)/2)-k}n^{((\ell-1)/2k)-1}} \\
                &= \frac{\log^kn}{n^{((\ell-1)/2k)-1/2}}
                \leq \frac{\log^kn}{n^{((k+1)/2k)-1/2}}
                = \frac{\log^kn}{n^{1/2k}}
                \leq \sqrt{\log n}.
            \end{align*}
        }
    \begin{align*}
        (pn)^{k} \|\mu_{\ell,w}-\mu_{\ell,w'}\| \lesssim \frac{\sqrt{\log n}}{\sqrt{pn}},
    \end{align*}
    as desired.
\end{proof}

\section{A hitting time formula}
It is easy to see that for $\op{dist}(w,v)\geq\ell$, we have
\[
    H_{wv} = \ell + H_{\mu_\ell v}.
\]
For $\op{dist}(w,v) < \ell$, this formula is not true in general, but the next lemma gives a precise formula for~$H_{wv}$ in terms of~$H_{\mu_\ell v}$, up to some small error term.

\begin{lemma}\label{lemma: formula for Hitting times}
    Let $\ell\in \bN$. Let $G\in\cG_{k,n,p}$ and $v\neq w\in V(G)$.
    Let~$X=X_0,X_1,\ldots$ be a random walk on~$G$ starting at~$w$. Then,
    \begin{equation*}
        H_{wv} = \ell + H_{\mu_\ell v} - \frac{2|E(G)|}{d(v)}\sum_{i=1}^{\ell-1} \pr_w[X_i=v] \pm \cO\left(\frac{1}{pn}\right).
    \end{equation*}
\end{lemma}

We need the following lemmas. First, \cref{lemma: formula for Hitting times lemma} yields a formula for the probability of the random walk returning to~$v$ for the first time after a certain number of steps. 
\begin{lemma}\label{lemma: formula for Hitting times lemma}
    Let $\ell\in \bN$ and~$G$ be a graph and $X=X_0,X_1,\ldots$ a random walk on~$G$. Then,
    \begin{equation*}
        \pr[X_\ell=v;X_{1},\ldots,X_{\ell-1}\neq v \mid X_0=v]
        = \frac{1}{|N(v)|}\sum_{u\in N(v)}\pr[T_{uv} = \ell-1].
    \end{equation*}
\end{lemma}

\begin{proof}
    We calculate
            \begin{align*}
                &\pr[X_\ell=v;X_{1},\ldots,X_{\ell-1}\neq v \mid X_0=v] \\
                &\qquad = \sum_{u\in N(v)} \pr[X_\ell=v;X_{1}=u; X_{2}\ldots,X_{\ell-1}\neq v \mid X_0=v] \\
                &\qquad = \sum_{u\in N(v)} \pr[X_{1}=u\mid X_0=v]\pr[X_\ell=v;X_{1},\ldots,X_{\ell-1}\neq v \mid X_{1}=u] \\
                &\qquad = \frac{1}{|N(v)|}\sum_{u\in N(v)} \pr[X_{\ell-1}=v;X_{0},\ldots,X_{\ell-2}\neq v \mid X_{0}=u] \\
                &\qquad = \frac{1}{|N(v)|}\sum_{u\in N(v)} \pr[T_{uv}=\ell-1],
            \end{align*}
            as desired.
\end{proof}

In the next lemma, we give a formula for the expected ``shifted" hitting time.
\begin{lemma}\label{lemma: index shifting}
    Let $\ell\in \bN_0$. Let~$G$ be a graph, $v\neq w\in V(G)$, and let $X=X_0,X_1,\ldots$ be a random walk on~$G$. Then,
    \begin{equation*}
        \sum_{m\geq0} m\pr[T_{wv} = m+\ell] = H_{wv} - \ell - \sum_{m=1}^{\ell-1} (m-\ell) \pr[T_{wv} = m].
    \end{equation*}
\end{lemma}
\begin{proof}
    Note that, as $v\neq w$, we have $\pr[T_{wv}=0]=0$. Therefore,
    \begin{align*}
        \sum_{m\geq0} m\pr[T_{wv}=m+\ell] &= \sum_{m\geq \ell}(m-\ell)\pr[T_{wv}=m] \\
        &= \sum_{m\geq 0} (m-\ell)\pr[T_{wv}=m] - \sum_{m=1}^{\ell-1} (m-\ell)\pr[T_{wv}=m] \\
        &= H_{wv} - \ell - \sum_{m=1}^{\ell-1} (m-\ell)\pr[T_{wv}=m],
    \end{align*}
    as desired.
\end{proof}

The next lemma shows that the probability that the random walk hits~$v$ for the first time at constant time~$i$ is basically the same as the probability that the random walk hits~$v$ at time~$i$.
\begin{lemma}\label{lemma: pr T_wv = i}
    Let $i\in\bN$. Let~$G\in\cG_{k,n,p}$ and $v\neq w\in V(G)$. Let $X=X_0,X_1,\ldots$ be a random walk on~$G$. Then,
    \begin{equation*}
        \pr[T_{wv}=i] = \pr_w[X_i=v] \pm \cO\left(\frac{1}{p^2n^2}\right).
    \end{equation*}
\end{lemma}
\begin{proof}
    \Cref{cor: properties of G(np)}\ref{item: Twv} and \cref{lemma: formula for Hitting times lemma} imply
    \begin{align*}
        \pr[T_{wv}=i]
        &= \pr_w[X_i=v] - \sum_{j=1}^{i-2} \pr_w[X_i=v, X_{j+1},\ldots,X_{i-1}\neq v, X_j=v] \\
        &= \pr_w[X_i=v] - \sum_{j=1}^{i-2} \pr_w[X_j=v]\pr[X_i=v,X_{j+1},\ldots,X_{i-1} \neq v \mid X_j=v] \\
        &= \pr_w[X_i=v] - \sum_{j=1}^{i-2} \pr_w[X_j=v]\pr[X_{i-j}=v,X_{1},\ldots,X_{i-j-1} \neq v \mid X_0=v] \\
        &= \pr_w[X_i=v] - \sum_{j=1}^{i-2} \pr_w[X_j=v]\frac{1}{|N(v)|}\sum_{u\in N(v)}\pr[T_{uv}=i-j-1] \\
        &= \pr_w[X_i=v] \pm \cO\left(\frac{1}{p^2n^2}\right),
    \end{align*}
    as desired.
\end{proof}

\begin{proof}[Proof of~\cref{lemma: formula for Hitting times}]
    We have
    \begin{align}\label{equation: H_muell_1}
        H_{\mu_\ell v} &= \sum_x \mu_\ell(x) H_{xv} = \sum_{m\geq0}m\sum_x\pr_w[X_\ell=x]\pr[T_{xv}=m].
    \end{align}
    For $m\geq 1$, we have that
    \begin{align*}
        \sum_x\pr_w[X_\ell=x]\pr[T_{xv}=m] &= \sum_x\pr_w[X_\ell=x]\pr[X_m=v;X_0,\ldots,X_{m-1}\neq v \mid X_0=x] \\
        &= \sum_x \pr_w[X_\ell=x]\pr_w[X_{m+\ell}=v;X_\ell,\ldots,X_{m+\ell-1}\neq v\mid X_\ell=x] \\
        &= \sum_x \pr_w[X_{m+\ell}=v;X_\ell=x;X_\ell,\ldots,X_{m+\ell-1}\neq v] \\
        &= \pr_w[X_{m+\ell}=v;X_\ell,\ldots,X_{m+\ell-1}\neq v].
    \end{align*}
    Inserting in~\eqref{equation: H_muell_1} and applying~\cref{lemma: formula for Hitting times lemma} yields
    \begin{align*}\label{equation: H_muellv}
        H_{\mu_\ell v} &= \sum_{m\geq0}m\pr_w[X_{m+\ell}=v;X_\ell,\ldots,X_{m+\ell-1}\neq v] \nonumber \\
        &= \sum_{m\geq0}m\pr[T_{wv}=m+\ell] + \sum_{m\geq0}m\sum_{i=1}^{\ell-1} \pr_w[X_{m+\ell} = v; X_i = v; X_{i+1},\ldots,X_{m+\ell-1}\neq v] \nonumber \\
        &= \sum_{m\geq0}m\pr[T_{wv}=m+\ell] \\
        &\qquad\qquad + \sum_{m\geq0}m\sum_{i=1}^{\ell-1} \pr_w[X_i=v] \pr_w[X_{m+\ell} = v; X_{i+1},\ldots,X_{m+\ell-1}\neq v \mid X_i=v] \nonumber \\
        &= \sum_{m\geq0}m\pr[T_{wv}=m+\ell]\\
        &\qquad \qquad + \sum_{m\geq0}m\sum_{i=1}^{\ell-1} \pr_w[X_i=v] \pr[X_{m+\ell-i} = v; X_1,\ldots,X_{m+\ell-i-1}\neq v \mid X_0=v] \nonumber \\
        &= \sum_{m\geq0}m\pr[T_{wv}=m+\ell] + \sum_{i=1}^{\ell-1} \pr_w[X_i=v] \frac{1}{|N(v)|}\sum_{u\in N(v)} \sum_{m\geq0}m \pr[T_{uv} = m+\ell-i-1].
    \end{align*}
    Applying~\cref{lemma: index shifting} and using \cref{cor: properties of G(np)}\cref{item: Twv}, we obtain
    \begin{align*}
        H_{\mu_\ell v} &= H_{wv} - \ell - \sum_{m=1}^{\ell-1} (m-\ell)\pr[T_{wv}=m] \\
        &\quad + \sum_{i=1}^{\ell-1} \pr_w[X_i=v] \frac{1}{|N(v)|}\sum_{u\in N(v)} \left( H_{uv} - (\ell-i-1) - \sum_{m=1}^{\ell-i-2} (m-(\ell-i-1)) \pr[T_{uv}=m] \right) \\
        &= H_{wv} - \ell - \sum_{i=1}^{\ell-1} (i-\ell)\pr[T_{wv}=i]
        + \sum_{i=1}^{\ell-1}\pr_w[X_i=v]\frac{1}{|N(v)|}\sum_{u\in N(v)} H_{uv} \pm \cO\left(\frac{1}{pn}\right).
    \end{align*}
    \Cref{lemma: pr T_wv = i} and \cref{cor: properties of G(np)}\cref{item: Twv} yield%
        \COMMENT{We have
            \begin{equation*}
                -\sum_{i=1}^{\ell-1} (i-\ell)\pr[T_{wv}=i]
                = -\sum_{i=1}^{\ell-1} (i-\ell+1)\pr_w[X_i=v] + \sum_{i=1}^{\ell-1} \pr_w[X_i=v] \pm \cO\left(\frac{1}{p^2n^2}\right) = \sum_{i=1}^{\ell-1}\pr_w[X_i=v] \pm \cO\left(\frac{1}{pn}\right).
            \end{equation*}}
    \begin{align*}
        H_{\mu_\ell v} &= H_{wv} - \ell
        + \sum_{i=1}^{\ell-1}\pr_w[X_i=v] \left(\frac{1}{|N(v)|}\sum_{u\in N(v)} H_{uv}+1\right)
        \pm\cO\left(\frac{1}{pn}\right).
    \end{align*}
    Finally, rearranging and using~\cref{lemma: H_nuv}, we obtain
    \begin{align*}
        H_{wv} &= \ell + H_{\mu_\ell v} - \frac{2|E(G)|}{d(v)}\sum_{i=1}^{\ell-1} \pr_w[X_i=v] \pm \cO\left(\frac{1}{pn}\right),
    \end{align*}
    as desired.
\end{proof}

\section{Calculation of hitting times}

In the following proposition, we compare~$H_{wv}$ to~$H_{uv}$ for any $u\neq v\neq w\in V(G)$.
\begin{prop}\label{prop: mixed hitting times}
    Let $G\in \cG_{k,n,p}$.
    Let $u\neq v\neq w\in V(G)$. Then,
    \begin{align*}
        H_{wv} -H_{uv} = \frac{2|E(G)|}{d(v)}\sum_{i=1}^{3k+1}(\pr_u[X_i=v]-\pr_w[X_i=v]) \pm \cO\left(\frac{\sqrt{\log n}}{\sqrt{pn}}\right).
    \end{align*}
\end{prop}
\begin{proof}
\Cref{lemma: formula for Hitting times} (applied with $\ell=3k+2$) gives
\begin{align*}
    H_{wv}-H_{uv} &= H_{\mu_{\ell, w}v}-H_{\mu_{\ell, u}v} - \frac{2|E(G)|}{d(v)} \sum_{i=1}^{3k+1} (\pr_w[X_i=v] - \pr_u[X_i=v]) \pm \cO\left(\frac{1}{pn}\right).
\end{align*}
Applying \cref{lemma: H_muv-Hmu'v} yields the desired result.
\end{proof}

We can now prove the following explicit formula for hitting times.
\begin{theorem}\label{thm: main version with probabilities}
    Let $G\in\cG_{k,n,p}$ and $v\neq w\in V(G)$. Then,
    \begin{align*}
        H_{wv} &= \frac{2|E(G)|}{d(v)} - 1 + \frac{2|E(G)|}{d(v)} \frac{1}{|N(v)|} \sum_{u\in N(v)}\sum_{i=1}^{3k+1}(\pr_u[X_i=v]-\pr_w[X_i=v]) \pm\cO\left(\frac{\sqrt{\log n}}{\sqrt{pn}}\right).
    \end{align*}
\end{theorem}
\begin{proof}
    From \cref{lemma: H_nuv}, we know that the average hitting time in~$N(v)$ satisfies
    \begin{equation}\label{equation: H_A}
        H_{N(v)} = \frac{1}{|N(v)|}\sum_{u\in N(v)} H_{uv} = \frac{2|E(G)|}{d(v)} - 1.
    \end{equation}
    \Cref{prop: mixed hitting times} implies that
    \begin{align*}
        H_{wv} - H_{N(v)} &= \frac{1}{|N(v)|}\sum_{u\in N(v)} (H_{wv}-H_{uv}) \\
        &= \frac{1}{|N(v)|}\sum_{u\in N(v)}\frac{2|E(G)|}{d(v)}\sum_{i=1}^{3k+1}(\pr_u[X_i=v]-\pr_w[X_i=v])\pm\cO\left(\frac{\sqrt{\log n}}{\sqrt{pn}}\right)
    \end{align*}
    and thus \cref{equation: H_A} gives the desired result.
\end{proof}

We want to express the second term in our formula in terms of number of walks starting from~$w$ and~$u$, respectively.
Recall that $\cW_i(w)$ denotes the number of walks of length~$i$ starting at~$w$ and $\cW_i(w,v)$ the number of $w$-$v$ walks of length~$i$.

\begin{prop}\label{prop: probabilities in terms of paths}
    Let $i\in\bN$ and $G\in\cG_{k,n,p}$.
    For all $v,w\in V(G)$, we have
    \begin{align*}
        \pr_w[X_i = v] = \frac{\cW_i(w,v)}{\cW_i(w)} \pm \cO\left(\frac{\sqrt{\log n}}{p^{3/2}n^{3/2}}\right).
    \end{align*}
\end{prop}
\begin{proof}
    The claim is certainly true for $i=1$. Assume it holds true for all $i'<i$. By \cref{cor: properties of G(np)}\cref{item: 1/N(v)}, we have
    \begin{align*}
        \pr_w[X_i = v] &= \sum_{u\in N(v)} \pr[X_i=v \mid X_{i-1} = u]\pr_w[X_{i-1}=u] \\
        &= \sum_{u\in N(v)} \frac{1}{|N(u)|} \pr_w[X_{i-1}=u] \\
        &= \left(\frac{1}{pn}\pm\cO\left(\frac{\sqrt{\log n}}{p^{3/2}n^{3/2}}\right)\right) \sum_{u\in N(v)} \pr_w[X_{i-1}=u].
    \end{align*}
    Applying the induction hypothesis yields $\pr_w[X_{i-1} = u] = \frac{\cW_{i-1}(w,u)}{\cW_{i-1}(w)}\pm\cO(\frac{\sqrt{\log n}}{p^{3/2}n^{3/2}})$ for all $u\in V(G)$. 
    Note that \cref{lemma: properties of G(np)}\cref{item: N(v)} implies that $\cW_i(w)=(pn\pm \cO(\sqrt{pn\log n}))\cW_{i-1}(w)$ and that any walk starting at~$w$ of length $i-1$ ending in~$N(v)$ can be extended to exactly one $w$-$v$ walk of length~$i$ and at least $\frac{pn}{2}$ distinct walks of length~$i$ starting at~$w$. This implies that $\frac{\cW_i(w,v)}{\cW_i(w)}=\cO(\frac{1}{pn})$.
    Thus, we obtain
    \begin{align*}
        \pr_w[X_i = v]
        &= \left(\frac{1}{pn}\pm\cO\left(\frac{\sqrt{\log n}}{p^{3/2}n^{3/2}}\right)\right) \sum_{u\in N(v)} \frac{\cW_{i-1}(w,u)}{\cW_{i-1}(w)}\pm |N(v)|\cdot\cO\left(\frac{\sqrt{\log n}}{p^{5/2}n^{5/2}}\right) \\
        &= \left(\frac{1}{pn}\pm\cO\left(\frac{\sqrt{\log n}}{p^{3/2}n^{3/2}}\right)\right)\frac{\cW_i(w,v)}{\cW_{i-1}(w)} \pm \cO\left(\frac{\sqrt{\log n}}{p^{3/2}n^{3/2}}\right) \\
        &= \left(\frac{1}{pn}\pm\cO\left(\frac{\sqrt{\log n}}{p^{3/2}n^{3/2}}\right)\right)(pn\pm \cO(\sqrt{pn\log n}))\frac{\cW_i(w,v)}{\cW_i(w)} \pm \cO\left(\frac{\sqrt{\log n}}{p^{3/2}n^{3/2}}\right) \\
        &= \left(1\pm\cO\left(\frac{\sqrt{\log n}}{\sqrt{pn}}\right)\right)\frac{\cW_i(w,v)}{\cW_i(w)} \pm \cO\left(\frac{\sqrt{\log n}}{p^{3/2}n^{3/2}}\right) \\
        &= \frac{\cW_i(w,v)}{\cW_i(w)} \pm \cO\left(\frac{\sqrt{\log n}}{p^{3/2}n^{3/2}}\right),
    \end{align*}
    as desired.
\end{proof}

We are now able to prove~\cref{thm: main}.

\begin{proof}[Proof of~\cref{thm: main}]
Let $k\geq 2$ and $n$ be integers, let $\frac{\log n}{n^{(k-1)/k}}\leq p\leq 1-\Omega(\frac{\log^4n}{n})$.
    By \cref{lemma: properties of G(np)}, any $G\sim G(n,p)$ belongs to $\cG_{k,n,p}$ with high probability. Thus, it suffices to consider arbitrary $G\in \cG_{k,n,p}$ and $v\neq w\in V(G)$, and show that the hitting time $H_{wv}$ satisfies (deterministically) the desired formula.
By~\cref{thm: main version with probabilities} we obtain
\begin{align*}
    H_{wv} &= \frac{2|E(G)|}{d(v)} - 1 + \frac{2|E(G)|}{d(v)} \frac{1}{|N(v)|} \sum_{u\in N(v)}\sum_{i=1}^{3k+1}(\pr_u[X_i=v] -\pr_w[X_i=v])\pm\cO\left(\frac{\sqrt{\log n}}{\sqrt{pn}}\right).
\end{align*}
Applying~\cref{prop: probabilities in terms of paths} and~\cref{cor: properties of G(np)}\cref{item: 2E/deg v} yields the desired result.
\end{proof}

To prove~\cref{cor: diameter 2}, we need the following lemma, which estimates the probability that a random walk hits~$v$ after~$\ell$ steps.
\begin{lemma}\label{lemma: random walk mixing time}
    Let $k\geq 2$ and $\ell\geq 3$. Let $p\geq \frac{\log n}{\sqrt{n}}$ and $G\in \cG_{k,n,p}$. Let $X=X_0,X_1,\ldots$ be a random walk on~$G$ starting at $w\in V(G)$. Then,
    \begin{equation*}\label{equation: pr_w[X_2=x]}
        \pr_w[X_2=x] = \left\{\begin{matrix*}[l]
        \frac{1}{n} \pm \cO\left(\frac{\sqrt{\log n}}{pn^{3/2}}\right) & \text{if }x\neq w \\
        \frac{1}{pn} \pm \cO\left(\frac{\sqrt{\log n}}{p^{3/2}n^{3/2}}\right) & \text{if }x=w,
        \end{matrix*}\right.
    \end{equation*}
    and
    \begin{equation*}\label{equation: pr_w[X_k=x]}
        \pr_w[X_\ell=x] = \frac{1}{n} \pm \cO\left(\frac{\sqrt{\log n}}{pn^{3/2}}\right).
    \end{equation*}
\end{lemma}
\begin{proof}
Using \cref{cor: properties of G(np)}\cref{item: 1/N(v)} we obtain
\begin{align*}
    \pr_w[X_2=w] &= \sum_{u\in N(w)} \pr[X_2=w\mid X_1=u]\pr_w[X_1=u]
    = \frac{1}{|N(w)|}\sum_{u\in N(w)} \frac{1}{|N(u)|}\\
    &= \frac{1}{pn} \pm \cO\left(\frac{\sqrt{\log n}}{p^{3/2}n^{3/2}}\right),
\end{align*}
    while for $x\neq w$, \cref{lemma: properties of G(np)}\ref{item: N(v) cap N(w)}, \cref{cor: properties of G(np)}\ref{item: 1/N(v)}, and the fact that $p\geq \frac{\log n}{\sqrt{n}}$ imply
        \COMMENT{
            We have
            \begin{align*}
                \left(\frac{1}{pn}\pm\cO\left(\frac{\sqrt{\log n}}{p^{3/2}n^{3/2}}\right)\right)^2(p^2n\pm\cO(\sqrt{p^2n\log n}))
                &= \left(\frac{1}{p^2n^2}\pm\cO\left(\frac{\sqrt{\log n}}{p^{5/2}n^{5/2}}\right)\right)(p^2n\pm\cO(\sqrt{p^2n\log n})) \\
                &= \frac{1}{n} \pm \cO\left(\frac{\sqrt{\log n}}{\sqrt{p}n^{3/2}}\right) \pm \cO\left(\frac{\sqrt{\log n}}{pn^{3/2}}\right).
            \end{align*}
        }
    \begin{align*}
        \pr_w[X_2=x] &= \pr_w[X_2=x\mid X_1\in N(x)]\pr_w[X_1\in N(x)] \\
        &= \frac{1}{|N(X_1)|}\frac{|N(x)\cap N(w)|}{|N(w)|}\\
        &=\left(\frac{1}{pn}\pm \cO\left(\frac{\sqrt{\log n}}{p^{3/2}n^{3/2}}\right)\right)^2(p^2n\pm \cO(\sqrt{p^2n \log n}))\\
        &= \frac{1}{n} \pm \cO\left(\frac{\sqrt{\log n}}{pn^{3/2}}\right).
    \end{align*}
    For $\pr_w[X_3=x]$, we use \cref{lemma: properties of G(np)}\cref{item: N(v)}, \cref{cor: properties of G(np)}\cref{item: 1/N(v)}, and the above to obtain that
    \begin{align*}
        \pr_w[X_3=x] &= \sum_{u\in N(x)} \pr_w[X_3=x \mid X_2=u]\pr_w[X_2=u] = \sum_{u\in N(x)} \frac{1}{|N(u)|}\pr_w[X_2=u] \\
        &= \left(\frac{1}{pn}\pm \cO\left(\frac{\sqrt{\log n}}{p^{3/2}n^{3/2}}\right)\right) \left(\sum_{u\in N(x)\setminus\{w\}}\pr_w[X_2=u] \pm \pr_w[X_2=w]\right)\\
        &= \left(\frac{1}{pn}\pm \cO\left(\frac{\sqrt{\log n}}{p^{3/2}n^{3/2}}\right)\right)\left(\sum_{u\in N(x)\setminus\{w\}}\left(\frac{1}{n}\pm \cO\left(\frac{\sqrt{\log n}}{pn^{3/2}}\right)\right)\pm\cO\left(\frac{1}{pn}\right)\right)\\
        &= \left(\frac{1}{pn}\pm \cO\left(\frac{\sqrt{\log n}}{p^{3/2}n^{3/2}}\right)\right)(pn\pm \cO(\sqrt{pn\log n}))\left(\frac{1}{n}\pm \cO\left(\frac{\sqrt{\log n}}{pn^{3/2}}\right)\right) \pm \cO\left(\frac{1}{p^2n^2}\right) \\
        &= \frac{1}{n}\pm\cO\left(\frac{\sqrt{\log n}}{pn^{3/2}}\right).
    \end{align*}
    Now suppose that $\ell\geq 4$ and $\pr_w[X_{\ell-1}=u] = \frac{1}{n}\pm\cO(\frac{\sqrt{\log n}}{pn^{3/2}})$ for all $u\in V(G)$. Then, \cref{lemma: properties of G(np)}\cref{item: N(v)} and \cref{cor: properties of G(np)}\cref{item: 1/N(v)} give
    \begin{align*}
        \pr_w[X_\ell=x] &= \sum_{u\in N(x)} \pr_w[X_\ell=x\mid X_{\ell-1}=u]\pr_w[X_{\ell-1}=u] \\
        &= (pn\pm\cO(\sqrt{pn\log n}))\left(\frac{1}{pn}\pm \cO\left(\frac{\sqrt{\log n}}{p^{3/2}n^{3/2}}\right)\right)\left(\frac{1}{n}\pm\cO\left(\frac{\sqrt{\log n}}{pn^{3/2}}\right)\right) \\
        &= \frac{1}{n} \pm\cO\left(\frac{\sqrt{\log n}}{pn^{3/2}}\right).
    \end{align*}
    The statement follows by induction.
\end{proof}

\begin{proof}[Proof of~\cref{cor: diameter 2}]
We derive the \lcnamecref{cor: diameter 2} from \cref{thm: main version with probabilities} as follows. 
If $\op{dist}(w,v)=1$, then by using \cref{cor: properties of G(np)}\cref{item: 1/N(v),item: 2E/deg v}, and~\cref{lemma: random walk mixing time} we obtain
    \begin{align*}
        \frac{2|E(G)|}{d(v)} &\frac{1}{|N(v)|}\sum_{u\in N(v)}\sum_{i=1}^{3k+1} (\pr_u[X_i=v]-\pr_w[X_i=v]) \\
        &= \frac{2|E(G)|}{d(v)} \frac{1}{|N(v)|} \sum_{u\in N(v)} \left(\pr_u[X_1=v]-\pr_w[X_1=v] \pm \cO\left(\frac{\sqrt{\log n}}{pn^{3/2}}\right)\right) \\
        &= \left(n\pm\cO\left(\frac{\sqrt{n\log n}}{\sqrt{p}}\right)\right) \cO\left(\frac{\sqrt{\log n}}{p^{3/2}n^{3/2}}\right)
        = \cO\left(\frac{\sqrt{\log n}}{p^{3/2}\sqrt{n}}\right),
    \end{align*}
    as desired.
    Similarly, for $\op{dist}(w,v)=2$, we have
    \begin{align*}
        \frac{2|E(G)|}{d(v)} &\frac{1}{|N(v)|}\sum_{u\in N(v)}\sum_{i=1}^{3k+1} (\pr_u[X_i=v]-\pr_w[X_i=v]) \\
        &= \frac{2|E(G)|}{d(v)} \frac{1}{|N(v)|}\sum_{u\in N(v)} \left(\pr_u[X_1=v] \pm \cO\left(\frac{\sqrt{\log n}}{pn^{3/2}}\right)\right) \\
        &= \left(n\pm\cO\left(\frac{\sqrt{n\log n}}{\sqrt{p}}\right)\right) \left(\frac{1}{pn} \pm \cO\left(\frac{\sqrt{\log n}}{p^{3/2}n^{3/2}}\right)\right)\\
        &= \frac{1}{p} \pm \cO\left(\frac{\sqrt{\log n}}{p^{3/2}\sqrt{n}}\right),
    \end{align*}
    as desired.
\end{proof}

\bibliographystyle{abbrv}
\bibliography{bibliography}

\APPENDIX{\appendix

\section{Second eigenvalue}\label{app:lambda2}

In this appendix, we show how to adapt the arguments of \cite{furedi1981eigenvalues} to show that \cref{lemma: properties of G(np)}\cref{item:lambda2} holds with high probability in our range of $p$. In fact, we prove the following.

\begin{lm}\label{lm:lambda2}
    Let $p(1-p)= \Omega(\frac{\log^4 n}{n})$ and $G\sim G(n,p)$. Then, with high probability we have $\max\{|\lambda_2|, |\lambda_n|\}=\cO(\sqrt{pn})$.
\end{lm}

We will need the following more general result about random matrices. This was originally proved in \cite{furedi1981eigenvalues} for constant variance $\sigma^2$ and used to prove the constant $p$ version of \cref{lm:lambda2}. Vu \cite{vu2007spectral} later observed that the methods hold for a larger range of $\sigma^2$.

\begin{lm}[{Vu \cite{vu2007spectral}}]\label{lm:Fueredi}
    Let $\sigma^2= \Omega(\frac{\log^4 n}{n})$.
    Let $\{m_{i,j}\in [-1,1]\mid 1\leq i\leq j\leq n\}$ be independent random variables.
    For each $1\leq i\leq j\leq n$, let $m_{j,i}\coloneqq m_{i,j}$. Define a random symmetric matrix $M\coloneqq (m_{i,j})_{i,j\in [n]}$ and denote by $\mu_1\geq \dots \geq \mu_n$ its eigenvalues. Suppose that $\bE m_{i,j} =0$ and $\Var(m_{i,j})\leq \sigma^2$ for all $1\leq i< j\leq n$. Then, with high probability,
    $\max_{i\in [n]}|\mu_i|=\cO(\sigma\sqrt{n})$.
\end{lm}

We remark that in the original statement of \cref{lm:Fueredi}, that is \cite[Theorem 1.4]{vu2007spectral}, one also requires $\bE m_{i,i}=0$ and $\Var(m_{i,i})\leq \sigma^2$ for all $i\in [n]$. However, this is not required in our case since we further assume $m_{i,j}\in [-1,1]$ for all $1\leq i\leq j\leq n$. Indeed, as mentioned in \cite{vu2007spectral}, changing all the diagonal entries of $M$ to $0$ can only change $\max_{i\in [n]}|\mu_i|$ by at most $1=o(\sigma\sqrt{n})$%
    \COMMENT{For symmetric matrices $M=(m_{i,j})_{i,j\in [n]}$, we have \[\max_{i\in [n]}|\mu_i|=\|M\|_2=\sup_{v\in \bR^n, \|v\|_2=1}|vMv^T|=\sup_{(v_1,\dots, v_n)\in \bR^n, \sum_{i\in [n]}v_i^2=1}\left|\sum_{j,k\in [n]}v_jv_k m_{jk}\right|.\]
    So changing the value of each $m_{i,i}\in [-1,1]$ to $0$ changes $\left|\sum_{j,k\in [n]}v_jv_k m_{jk}\right|$ by at most $\sum_{i\in [n]}v_i^2=1$.}%
    \COMMENT{Original statement thus gives, in our case, a bound of 
    \[\max_{i\in [n]}|\mu_i|=\sigma\sqrt{n}+\cO(\sqrt{\sigma}n^{1/4}\log n)+1.\]
    For $\sigma^2= \Omega(\frac{\log^4 n}{n})$, we have $\sqrt{\sigma}n^{1/4}\log n=\cO(\sigma \sqrt{n})$ and $1=o(\sigma \sqrt{n})$.}.

\begin{proof}[Proof of \cref{lm:lambda2}]
    Let $J$ be the matrix with all $1$ entries and define $M\coloneqq A-pJ$.
    Note that $M=(m_{i,j})_{i,j\in [n]}$ is a symmetric random matrix with $m_{i,j}=m_{j,i}\in [-p,1-p]\subseteq [-1,1]$ for all $i,j\in [n]$ such that $\{m_{i,j}\mid 1\leq i\leq j\leq n\}$ are independent. Moreover, $\mathbb{E} m_{i,j}=0$ and $\Var(m_{i,j})=p(1-p)=\Omega(\frac{\log^4n}{n})$ for all $i\neq j\in[n]$.
    Let $\mu_1\geq \dots \geq \mu_n$ be the eigenvalues of $M$. By \cite[Lemma 1 and 2]{furedi1981eigenvalues}, we have $\lambda_2\leq \mu_1$ and $\lambda_n\geq \mu_n$. Observe also that $0=\tr(A)=\sum_{i\in [n]}\lambda_i$, so $\lambda_n,\mu_n<0$. Thus, the result follows by applying \cref{lm:Fueredi} to $M$%
        \COMMENT{\Cref{lm:Fueredi} gives
        \begin{align*}
            \max\{|\lambda_2|,|\lambda_n|\}= \cO(\sqrt{p(1-p)n})=\cO(\sqrt{pn}).
        \end{align*}}.
\end{proof}

\section{Proof of Proposition \ref{prop: mixing property}}\label{app:mixing}

In this appendix, we prove \cref{prop: mixing property}. Note that the arguments are the same as in~\cite{ottolini2023concentration}. The only difference is that, for our range of $p$, we need to use more precise values for the spectrum of the adjacency matrix (see \cref{lemma: properties of G(np)}\cref{item:phi,item:lambda1,item:lambda2}). We include the details here nonetheless for completeness.

We will need the \emph{vector $\ell_2$-norm} and \emph{matrix spectral norm}. Given a vector $x=(x_1,\dots, x_k)$, let $\|x\|_2\coloneqq \sqrt{\sum_{i\in [k]}|x_i|^2}$ and given a symmetric real matrix $M$, let $\|M\|_2$ be the largest eigenvalue of $M$. (Also recall that the entries of a diagonal matrix are its eigenvalues.)
Given two vectors $x=(x_1,\dots, x_k)$ and $y=(y_1,\dots, y_k)$, we write $\langle x, y\rangle\coloneqq \sum_{i\in [k]}x_iy_i$.
Given a graph $G$, we denote by $D$ the \emph{degree matrix} of $G$, that is, the diagonal matrix with an entry $d(v)$ for each $v\in V(G)$, and note that $D^{-1}A$ is the transition matrix of any random walk on $G$.

\begin{lemma}[{\cite[Lemma 3]{ottolini2023concentration}}]\label{lemma: mixing property}
    Let $G\in\cG_{k,n,p}$. Let~$v$ be a vector whose entries average to~$0$. Then,
    \begin{equation*}
        \|vD^{-1}A\|_2 \lesssim \frac{\sqrt{\log n}}{\sqrt{pn}} \|v\|_2.
    \end{equation*}
\end{lemma}

\begin{proof}
    We first make the following observation.
    \begin{claim}\label{claim}
        We have 
        \[\frac{\| vA\|_2}{pn}\lesssim \frac{\sqrt{\log n}}{\sqrt{pn}}\| v\|_2.\]
    \end{claim}
    \begin{proofclaim}
    Write $v=(v_1,\dots, v_n)$ and recall that $\phi=(\phi_1, \dots, \phi_n)$ denotes the eigenvector associated to the largest eigenvalue $\lambda_1$ of $A$. By assumption, $\sum_{i\in [n]}v_i=0$, so together with the Cauchy-Schwarz inequality and \cref{lemma: properties of G(np)}\cref{item:phi}, we obtain
    \begin{align*}
        \langle\phi,v\rangle^2&=\left(\sum_{i\in [n]}\phi_i v_i\right)^2
        =\left(\sum_{i\in [n]} \frac{v_i}{\sqrt{n}}+\sum_{i\in [n]}\left(\phi_i-\frac{1}{\sqrt{n}}\right)v_i\right)^2\\
        &\leq \sum_{i\in [n]}\left(\phi_i-\frac{1}{\sqrt{n}}\right)^2\cdot \sum_{i\in [n]}v_i^2
        \lesssim \frac{\log^3 n}{pn \log^2(pn)}\|v\|_2^2.
    \end{align*}
    Recall that the eigenvectors of $A$ form an orthonormal basis.%
        \COMMENT{Let $\phi, \psi_2, \dots, \psi_n$ be the eigenvectors associated to $\lambda_1, \lambda_2, \dots, \lambda_n$. These form a basis of $\mathbb{R}^n$, so we can write $v=\langle\phi,v\rangle\phi+\sum_{i=2}^n\langle\psi_i,v\rangle \psi_i$.
        Thus,
        \begin{align*}
            vA=\langle\phi,v\rangle \phi A+\sum_{i=2}^n\langle\psi_i,v\rangle \psi_iA =\langle\phi,v\rangle\lambda_1 \phi+\sum_{i=2}^n\langle\psi_i,v\rangle \lambda_i \psi_i,
        \end{align*}
        and so
        \begin{align*}
            \| vA\|_2^2&=\langle vA, vA\rangle
            = (vA)(vA)^T= \langle\phi,v\rangle^2\lambda_1^2 \langle\phi,\phi\rangle+\sum_{i=2}^n\lambda_i^2\langle\psi_i,v\rangle^2 \langle\psi_i,\psi_i\rangle\\
            &\leq \langle\phi,v\rangle^2\lambda_1^2 \langle\phi,\phi\rangle+\lambda_2^2\sum_{i=2}^n\langle\psi_i,v\rangle^2 \langle\psi_i,\psi_i\rangle\\
            &\leq \langle\phi,v\rangle^2\lambda_1^2+\lambda_2^2\sum_{i=1}^n\langle\psi_i,v\rangle^2
            = \langle\phi,v\rangle^2\lambda_1^2+\lambda_2^2\|v\|_2^2.
        \end{align*}}
    Hence, \cref{lemma: properties of G(np)}\cref{item:lambda1,item:lambda2} imply that
    \begin{align*}
        \frac{\| vA\|_2}{pn}&\leq \frac{1}{pn}\sqrt{\lambda_1^2\langle\phi, v\rangle^2+\lambda_2^2\|v\|_2^2}
        \lesssim\frac{1}{pn}\sqrt{(pn)^2\cdot \frac{\log^3 n}{pn \log^2(pn)}\|v\|_2^2+pn\|v\|_2^2}\\
        &\lesssim \frac{\log^{3/2}n}{\sqrt{pn}\log(pn)}\|v\|_2
        \leq \frac{\log^{3/2}n}{\sqrt{pn}\log(n^{1/k}\log n)}\|v\|_2
        \lesssim\frac{\sqrt{\log n}}{\sqrt{pn}} \|v\|_2,
    \end{align*}
    as desired.
    \end{proofclaim}

    Let $D_1$ be the diagonal $n\times n$ matrix with entries $\frac{1}{pn}$. Let $D_2\coloneqq D^{-1}-D_1$ and observe that $D_2$ is the diagonal matrix with entries $\frac{1}{d(v)}-\frac{1}{pn}$.
    Note that, by \cref{cor: properties of G(np)}\cref{item: 1/N(v)}, we have
    \begin{align*}
        \|D_1\|_2 =\frac{1}{pn} \qquad \text{and} \qquad \|D_2\|_2=\cO\left(\frac{\sqrt{\log n}}{p^{3/2}n^{3/2}}\right).
    \end{align*}    
    Together with \cref{lemma: properties of G(np)}\cref{item:lambda1} and \cref{claim}, this implies that
    \begin{align*}
        \| vD^{-1}A\|_2&=\| v(D_1+D_2)A\|_2 \leq \| vD_1A\|_2+\| vD_2A\|_2\\
        &\leq \| D_1\|_2\| vA\|_2 +\| v\|_2 \|D_2\|_2\|A\|_2
        \lesssim \frac{\| vA\|_2}{pn}+\frac{\sqrt{\log n}}{p^{3/2}n^{3/2}}\| v\|_2\|A\|_2\\
        &\lesssim \frac{\sqrt{\log n}}{\sqrt{pn}}\| v\|_2,
    \end{align*}
    as desired.

\end{proof}

\begin{proof}[Proof of \cref{prop: mixing property}]
    First, observe that the Cauchy-Schwarz inequality implies that $\|\mu_\ell-\pi\|^2\leq n\|\mu_\ell-\pi\|_2^2$, so it suffices to show that
    \[\|\mu_\ell-\pi\|_2\lesssim \frac{(\log n)^{(\ell-1)/2}}{(pn)^{\ell/2}}.\]
    We proceed by induction on $\ell$.
    By \cref{lemma: properties of G(np)}\cref{item: N(v)}, \cref{cor: properties of G(np)}\cref{item: 1/N(v),item: stationary distribtution}, we have
    \[\|\mu_1-\pi\|_2\leq \|\mu_1\|_2+\|\pi\|_2\lesssim \sqrt{\frac{pn}{(pn)^2}}+\sqrt{\frac{n}{n^2}}\lesssim \frac{1}{\sqrt{pn}}.\]
    For the induction step, suppose that $\|\mu_\ell-\pi\|_2\lesssim \frac{(\log n)^{(\ell-1)/2}}{(pn)^{\ell/2}}$.
    Observe that $\mu_\ell$ and $\pi$ are both probability distributions, so $\mu_\ell-\pi$ has mean value $0$. Thus, \cref{lemma: mixing property} implies that
    \begin{align*}
        \|\mu_{\ell+1}-\pi\|_2=\|\mu_\ell D^{-1}A-\pi\|_2=\|(\mu_\ell-\pi)D^{-1}A\|_2\lesssim\frac{(\log n)^{\ell/2}}{(pn)^{(\ell+1)/2}},
    \end{align*}
    as desired.
\end{proof}

}

\end{document}